\documentclass[11pt]{amsart}
\usepackage{graphicx,color}
\usepackage{latexsym}
\usepackage{graphicx} 
\usepackage{amsthm,amsmath, amssymb,amsfonts,esint}

\usepackage{amsmath,amssymb,amsfonts,amsthm}
\usepackage{layout,textcomp}

\theoremstyle{plain}

\theoremstyle{plain}
\newtheorem{theorem}{Theorem} [section]
\newtheorem{corollary}[theorem]{Corollary}
\newtheorem{lemma}[theorem]{Lemma}
\newtheorem{proposition}[theorem]{Proposition}

\theoremstyle{definition}

\theoremstyle{remark}
\newtheorem{remark}[theorem]{Remark}

\DeclareMathOperator{\di}{div}

\numberwithin{theorem}{section}
\numberwithin{equation}{section}
\numberwithin{figure}{section}

\newcommand{\pl}{\left(}
\newcommand{\pr}{\right)}
\newcommand{\Ric}{\text{\rm Ric}}  
\newcommand{\Sph}{\mathbb{S}}
\newcommand{\R}{\mathbb{R}}
\newcommand{\D}{\nabla}
\newcommand{\He}{\text{\rm Hess}}
\newcommand{\la}{\langle}
\newcommand{\ra}{\rangle}
\newcommand{\Si}{\Sigma}

\newcommand{\dsigma}{\text{\rm d}\sigma}
\newcommand{\dt}{\frac{\text{\rm d}}{\text{\rm d}t}}
\newcommand{\dr}{\frac{\text{\rm d}}{\text{\rm d}r}}
\newcommand{\ddt}{\frac{\text{\rm d}^2}{\text{\rm d}t^2}}

\newcommand{\rhobt}{\overline{\rho_t}}

\begin{document}

\title[Hawking mass and rigidity of minimal two-spheres]{Hawking mass and local rigidity of minimal two-spheres in three-manifolds}
\author[Davi M{\'{a}}ximo]{Davi M{\'{a}}ximo}
\address{Department of Mathematics, The University of Texas, 1 University Station C1200, Austin, TX 78712}
\email{maximo@math.utexas.edu}
\author[Ivaldo Nunes]{Ivaldo Nunes}
\address{Instituto Nacional de Matem\'{a}tica Pura e Aplicada - IMPA, Estrada Dona Castorina 110, 22460-320, Rio de Janeiro-RJ, Brazil}
\email{ivaldo82@impa.br}

\date{}
\maketitle
\begin{abstract}
We study rigidity of minimal two-spheres $\Sigma$ that locally maximize the Hawking mass on a Riemannian three-manifold with a positive lower bound on its scalar curvature. After assu\-ming strict stability of $\Sigma$, we prove that a neighborhood of it in $M$ is isometric to one of the deSitter-Schwarzschild metrics on $(-\epsilon,\epsilon)\times \Sigma$. We also show that if $\Sigma$ is a critical point for the Hawking mass on the deSitter-Schwarzschild manifold $\mathbb{R}\times\Sph^2$ and can be written as a graph over a slice $\Sigma_r=\{r\}\times\mathbb{S}^2$, then $\Sigma$ itself must be a slice,  and moreover that slices are indeed local maxima amongst competitors that are graphs with small $C^2$-norm.
\end{abstract}


\section{Introduction}
In the last decades, stable minimal surfaces have proven to be a very important tool in studying Riemannian three-manifolds $(M,g)$ with scalar curvature bounded below. 

It was Schoen and Yau who first observed in \cite{SY2} that the second variation formula of area provides an interesting interplay between the scalar curvature of a three-manifold $(M,g)$ and the total curvature of a stable minimal surface $\Sigma\subset M$, which in turn is related to the topology of $\Sigma$. As a consequence, they proved that if $(M,g)$ has nonnegative scalar curvature and $\Sigma$ is two-sided and compact, then either $\Sigma$ is a two-sphere or a totally geodesic flat two-torus.

Motivated by the above result of Schoen and Yau, Cai and Galloway \cite{CG} later showed that if $(M,g)$ is a three-manifold with nonnegative scalar curvature and $\Sigma$ is an embedded two-torus which is locally of least area (which is a condition stronger than stability), then $\Sigma$ is flat and totally geodesic, and $M$ splits isometrically as a product $(-\epsilon,\epsilon)\times\Sigma$ in a neigborhood of $\Sigma$. Recently, the correspoding rigidity result in the case where $\Sigma$ is either a two-sphere or a compact Riemann surface of genus greater than $1$ were proved by Bray, Brendle, and Neves \cite{BBN} and by the second author \cite{N}, respectively. We note that Micallef and Moraru \cite{MM} later found an alternative argument to prove these splitting results. Moreover, a similar rigidity result for area-minimizing projective planes was obtained in \cite{BBEN}.

Perhaps one of the most important uses of the above relation between the scalar curvature of a three-manifold $(M,g)$ and the total curvature of a stable minimal surface $\Sigma\subset M$ was in the proof of the positive mass theorem given by Schoen and Yau \cite{SY1}. The positive mass theorem is a fundamental result which relates Riemannian geometry and general relativity. It states that the ADM mass of a complete asymptotically flat three-manifold $(M,g)$ with nonnegative scalar curvature is always nonnegative and is only zero when $M$ is isometric to the flat Euclidean space $\mathbb{R}^3$. Later, Witten \cite{W} gave an independent proof of the positive mass theorem using spin methods.

Another important result in the context of general relativity which involves minimal surfaces and scalar curvature is the Penrore inequality proved by Huisken and Ilmanen \cite{HI}, and independently, by Bray \cite{B}. It states that if $(M,g)$ is a complete asymptotically flat three-manifold with nonnegative scalar curvature and boundary $\Sigma=\partial M\neq\emptyset$ being an outermost minimal two-sphere, then the ADM mass of $M$ is greater than or equal to the Haw\-king mass of $\Sigma$, with equality attained if, and only if, $M$ is isometric to the one-half of the Schwarzschild metric on $\mathbb{R}^3\setminus\{0\}$. We recall that the {\it Hawking mass} of a compact surface $\Sigma\subset(M,g)$, denoted by $m_\text{\rm H}(\Sigma)$, is defined as 
\begin{equation}\label{hawkingmass}
m_\text{\rm H}(\Sigma)=\pl\dfrac{|\Sigma|}{16\pi}\pr^{1/2}\pl1-\dfrac{1}{16\pi}\int_\Sigma H^2\,\dsigma-\dfrac{\Lambda}{24\pi}\,|\Sigma|\pr,
\end{equation}
where $H$ is the mean curvature of $\Sigma$ and $\Lambda=\inf_M R$.

The Schwarzschild metrics on $\mathbb{R}^3\setminus\{0\}$ can be seen as complete rotationally symmetric metrics on $\mathbb{R}\times\mathbb{S}^2$ with zero scalar curvature and the slice $\Sigma_0=\{0\}\times\mathbb{S}^2$ being the outermost minimal two-sphere. Each Schwarzschild metric is determined by the Hawking mass of $\Sigma_0$.
These metrics appear as spacelike slices of the Schwarzschild vacuum spacetime in general relativity.

Another class of metrics on $\mathbb{R}\times\mathbb{S}^2$ is the deSitter-Schwarzschild metrics. These metrics are complete periodic rotationally symmetric metrics on $\mathbb{R}\times\mathbb{S}^2$ with constant positive scalar curvature, and have $\Sigma_0=\{0\}\times\mathbb{S}^2$ as a stricly stable minimal two-sphere. They appear as spacelike slices of the deSitter-Schwarzschild spacetime, which is a solution to the vacuum Einstein equation with a positive cosmo\-lo\-gi\-cal cons\-tant. The deSitter-Schwarzschild metrics constitute a one-parameter family of metrics $\{g_a\}_{a\in(0,1)}$ and, in this work, we scale each $g_a$ to have scalar curvature equal to $2$ (see Section 2 for a more detailed description). 

In the present paper, we prove some results concerning the deSitter-Schwarzschild metrics $g_a$ on $\mathbb{R}\times\mathbb{S}^2$.  

We begin by considering the general situation of a two-sided closed surface $\Sigma$ which is a critical point of the Hawking mass on a three-manifold $(M,g)$ with $R\geqslant2$.  By writing the Euler-Lagrange equation of the mass (see Proposition \ref{firstvariation} of the Appendix), we prove that whenever $\Sigma$ has nonnegative mean curvature then it must be minimal or umbilic with $R=2$ along $\Sigma$ and constant Gauss curvature. 

In particular, whenever $M$ is the deSitter-Schwarzschild manifold $(\mathbb{R}\times\mathbb{S}^2, g_a)$, the above says that critical points of the Hawking mass are either minimal surfaces or  slices $\{ r\} \times \mathbb{S}^2$. 

\begin{remark}\label{remark0}
To the best of our knowledge, there is no complete classification of minimal surfaces in $(\mathbb{R}\times\mathbb{S}^2, g_a)$  to be found in the literature. However, when the minimal  surface $\Sigma$ can be written as a graph over a slice $\Sigma_r$, then a result of Montiel \cite{Mo} says that $\Sigma$ must itself be a slice. See also \cite{Br}.
\end{remark}

The above considerations are evidence that local maxima of the Hawking mass in $(\mathbb{R}\times\mathbb{S}^2, g_a)$ must be slices. In our first main result we show that slices are indeed local maxima in the following sense: 

\begin{theorem}\label{theorem1}
Let $\,\Sigma_r=\{r\}\times\mathbb{S}^2$ be a slice of the deSitter-Schwarszchild manifold $(\mathbb{R}\times\mathbb{S}^2,g_a)$. Then there exists an $\epsilon=\epsilon(r)>0$ such that if ${\Sigma}\subset\mathbb{R}\times\mathbb{S}^2$ is an embedded two-sphere which is a normal graph over $\Sigma_r$ given by $\varphi\in C^{2}(\Sigma_r)$ with $\left\|\varphi\right\|_{C^{2}(\Sigma_r)}<\epsilon$, one has
\begin{itemize}
\item[(i)] either $m_\text{\rm H}({\Sigma})<m_\text{\rm H}(\Sigma_r)$;
\item[(ii)] or ${\Sigma}$ is a slice $\Sigma_s$ for some $s$.
\end{itemize}
\end{theorem}


The proof follows by showing that the second variation of the mass at each slice is strictly negative, unless the variation has constant speed, and using this to argue mini\-mality amongst surfaces that are graphs with small $C^2$ norm over the slice $\Sigma_r$. 

\begin{remark}
A general formula for the second variation on an arbitrary three-manifold is given in Proposition \ref{secondvariation} of the Appendix.  
\end{remark}

Our second result is a  local rigidity for the  deSitter-Schwarzschild mani\-fold $(\mathbb{R}\times\mathbb{S}^2,g_a)$ which involves strictly stable minimal surfaces and the Hawking mass. We prove:


\begin{theorem}\label{theorem2}
Let $(M,g)$ be a Riemannian three-manifold with scalar curvature $R\geqslant 2$. If $\Sigma\subset M$ is an embedded strictly stable minimal two-sphere which locally maximizes the Hawking mass, then the Gauss curvature of $\Sigma$ is constant equal to $1/a^2$ for some $a\in(0,1)$ and a neighborhood of $\Sigma$ in $(M,g)$ is isometric to the deSitter-Schwarzschild metric $((-\epsilon,\epsilon)\times\Sigma, g_a)$ for some $\epsilon>0.$ 

\end{theorem}


The idea of the proof goes as follows. Let $\lambda_1(\Sigma)$ denote the first eigenvalue of the Jacobi operator. The first step is to prove an infini\-tesimal rigidity along $\Sigma$ which is  obtained as follows.  Using the fact that $\Sigma$ is strictly stable we get an upper bound of the form
\begin{equation}\label{aux}
(1+\lambda_1(\Sigma))|\Sigma|\leqslant 4\pi.
\end{equation}
 On the other hand, the fact that $\Sigma$ locally maximizes the Hawking mass implies $\eqref{aux}$ with opposite inequality sign. Therefore equality is achieved and from it the infi\-ni\-te\-simal rigidity is attained.
  
 From this infinitesimal rigidity we next are able to construct a {\it constant mean curvature} ({CMC})  foliation of a neighborhood of $\Sigma$ by embedded two-spheres  $\{\Sigma(t)\subset M\}_{t\in(-\epsilon,\epsilon)}$, where $\Sigma_0=\Sigma$. Finally, by using the pro\-per\-ties of the CMC foliation $\Sigma(t)$ we obtain, decreasing $\epsilon$ if necessary, a monotonicity of the Hawking mass along $\Sigma(t)$. In particular, we get that $m_\text{\rm H}(\Sigma(t))\geqslant m_\text{\rm H}(\Sigma)$
 for all $t\in(-\epsilon,\epsilon)$. The rigidity result then follows from this.

Some remarks are now in order. First we point out that the upper bound \eqref{aux} involving $\lambda_1(\Sigma)$ is sharp and is achieved on strictly minimal slices in the deSitter-Schwarzschild manifold $(\mathbb{R}\times\mathbb{S}^2,g_a)$. In case $\Sigma$ is stable with $\lambda_1(L)=0$, \eqref{aux} is the area bound that appear in \cite{BBN}, which is attained on slices of the standard cylinder $(\mathbb{R}\times\mathbb{S}^2, \text{\rm d}r^2 + g_{\Sph^2})$.

 Moreover, we note that $g_a$ tends to $\text{\rm d}r^2 + g_{\Sph^2}$ when  $a\rightarrow1$, so it is interesting to ask whether the rigidity statement in \cite{BBN} can be proven under the same hypothesis of Theorem \ref{theorem2} but with strict stability replaced by stability. It turns out this is not the case as one can construct examples of three-manifolds with scalar curvature $R\geq 2$ that are not the standard cylinder and that contain a minimal two-sphere $\Sigma$ with area equal to $4\pi$, see $e.g.$ page 2 of \cite{MM}. It is then straightforward to check that a minimal two-sphere with area $4\pi$ is a global maximum of the Hawking mass.
 
\begin{remark}
The use of a CMC foliation in the proof above is inspired by the work in \cite{Bray} and \cite{BBN} (see also \cite{N}, \cite{MM})
\end{remark}


\section{Preliminaries}

In this section, we start defining the deSitter-Schwarzschild manifold. The deSitter-Schwarzschild metric with mass $m>0$ and scalar curvature equal to 2 is the metric
\[
\pl1-\dfrac{r^2}{3}-\dfrac{2\,m}{r}\pr^{-1}dr^2+r^2g_{\mathbb{S}^2}
\]
defined on $(a_0,b_0)\times\mathbb{S}^2$, where $(a_0,b_0)=\left\{r>0:1-\frac{r^2}{3}-\frac{2\,m}{r}
>0\right\}$ and $g_{\mathbb{S}^2}$ is the standard metric on $\mathbb{S}^2$ with constant Gauss curvature equal to $1$. 

The deSitter-Schwarzschild metric above extends to a smooth metric $g$ on $[a_0,b_0]\times\mathbb{S}^2$ and the boundary components $\{a_0\}\times\mathbb{S}^2$ and $\{b_0\}\times\mathbb{S}^2$ are totally geodesic two-spheres with respect to the metric $g$. Thus, after reflection along $\{a_0\}\times\mathbb{S}^2$, we find a smooth metric $g$ on $[c_0,b_0]\times\mathbb{S}^2$, where $c_0=2a_0-b_0.$ Since $\{c_0\}\times\mathbb{S}^2$ and $\{b_0\}\times\mathbb{S}^2$ are totally geodesic two-spheres, we can use the metric $g$ on $[c_0,b_0]\times\mathbb{S}^2$ to define, by reflection,  a complete periodic rotationally symmetric metric on $\mathbb{R}\times\mathbb{S}^2$ with scalar curvature equal to $2$. This metric is called the \textit{deSitter-Schwarszchild metric} with mass $m>0$ and scalar curvature equal to 2 on $\mathbb{R}\times\mathbb{S}^2$.

In order to deal with this metric in our paper, we will use the warped product notation. More precisely, consider the warped product metric $g=\textrm{d}r^2+u(r)^2g_{\mathbb{S}^2}$ on $\mathbb{R}\times\mathbb{S}^2$, where $u(r)$ is a positive real function. If we assume that $g$ has constant scalar curvature equal to 2, then $u$ solves the following second-order differential equation
\begin{equation}\label{EDO1}
u^{\prime\prime}(r)=\dfrac{1}{2}\pl\dfrac{1-u^\prime(r)^2}{u(r)}\pr-\dfrac{u(r)}{2}.
\end{equation}

Considering only positive solutions $u(r)$ to \eqref{EDO1} which are defined for all $r\in\mathbb{R}$, we get a one-parameter family of periodic rotationally symmetric metrics $g_a=\textrm{d}r^2+u_a(r)^2g_{\mathbb{S}^2}$ with constant scalar curvature equal to $2$, where $a\in(0,1)$ and $u_a(r)$ satisfies $u_a(0)=a=\min u$ and $u^{\prime}_a(0)=0$. These metrics are precisely the deSitter-Schwarzschild metrics on $\mathbb{R}\times\mathbb{S}^2$ defined above.


\begin{remark}\label{remark1}
Note that when $a$ tends to $1$, the metric $g_a$ tends to the standard product metric $\textrm{d}r^2+g_{\mathbb{S}^2}$ on $\mathbb{R}\times\mathbb{S}^2$. Moreover, observe that $\Sigma_0=\{0\}\times\mathbb{S}^2$ is a strictly stable minimal (in fact, totally geodesic) two-sphere of area $4\pi a^2$ in $(\mathbb{R}\times\mathbb{S}^2,g_a)$, for each $a\in(0,1)$, but  in the standard product metric $\textrm{d}r^2+g_{\mathbb{S}^2}$ on $\mathbb{R}\times\mathbb{S}^2$, i.e, in the limit as $a\rightarrow 1$, $\Sigma_0$ is only stable and not strictly so.
\end{remark}




\begin{remark}\label{remark2}
It follows from the first variation formula of the Hawking mass (see Appendix) that if a two-sphere $\Sigma\subset M$ is umbilic and has constant Gauss curvature and $M$ has constant scalar curvature equal to 2 along $\Sigma$, then $\Sigma$ is a critical point of the Hawking mass. 
\end{remark}

Denote by $\Sigma_r$ the slice $\{r\}\times\mathbb{S}^2$. By Remark \ref{remark2}, $\Sigma_r$ is a critical point for the Hawking mass in $(\mathbb{R}\times\mathbb{S}^2,g_a)$, for all $r\in\mathbb{R}$ and $a\in(0,1)$. Moreover, we note that the Hawking mass of $\Sigma_r\subset(\mathbb{R}\times\mathbb{S}^2,g_a)$ is constant for all $r\in\mathbb{R}$. It follows by a straightforward computation:
\[
\dr m_\text{\rm H}(\Sigma_r)=\dfrac{1}{2}u^\prime(r)(1-u^\prime(r)-u(r)^2-2u(r)u^{\prime\prime}(r)),
\]
which is zero once $u(r)$ solves \eqref{EDO1}, we obtain therefore that $m_\text{\rm H}(\Sigma_r)$ is constant equal to $m_\text{\rm H}(\Sigma_0)$. We will denote by $m_a$ this constant value. Thus, in what follows, $g_a$ is the deSitter-Schwarzschild metric with mass $m_a$ and scalar curvature equal to $2$ on $\mathbb{R}\times\mathbb{S}^2$.


\section{Proof of Theorem \ref{theorem1}}

We establish the following proposition before going into the proof of Theo\-rem \ref{theorem1}.


\begin{proposition}\label{proposition1}
Let $(M,g)$ be a Riemannian three-manifold with scalar curvature $R\geqslant 2$. If a two-sided closed surface $\Sigma\subset M$ with nonnegative mean curvature is a critical point of the Hawking mass, then $\Sigma$ is minimal or $\Sigma$ is umbilic, $R=2$ along $\Sigma$, and $\Sigma$ has constant Gauss curvature.
\end{proposition}
\begin{proof}
Let $H$ be the mean curvature of $\Sigma$. By Proposition \ref{firstvariation}, the Euler-Lagrange equation for the Hawking mass functional is 
\begin{equation}\label{EL}
\Delta_\Sigma H+QH=0,
\end{equation}
where
\[
Q=\dfrac{4\pi}{|\Sigma|}-K_\Sigma+\dfrac{1}{2}(R-2)+\dfrac{1}{4}\pl2|A|^2-\dfrac{1}{|\Sigma|}\int_\Sigma H^2\,\dsigma\pr
\]
and satisfies the condition:
\[
\int_\Sigma Q\,\dsigma\geqslant 0,
\]
with equality if, and only if, $\Sigma$ is umbilic.

Since $H\geqslant 0$, we can apply the maximum principle to \eqref{EL} to obtain that either $H\equiv 0$ or $H>0$.

Now, suppose that $H>0$. In this case, by \eqref{EL}, we have that
\[
\dfrac{1}{H}\Delta_\Sigma H+Q=0,
\]
which we integrate by parts and get 
\[
0=\int_\Sigma\dfrac{|\D H|^2}{H^2}\,\dsigma+\int_\Sigma Q\,\dsigma\geqslant0.
\]

Thus, $\int_\Sigma Q\,\dsigma=0$ which implies that $\Sigma$ is umbilic and $R$ is constant equal to $2$ along $\Sigma$. Moreover, we also get that $H$ is a constant function. So, we conclude that $Q=0$. Since $\Sigma$ is umbilic and $R$ is constant along $\Sigma$, we also obtain that Gauss curvature of $\Sigma$ is constant equal $\frac{4\pi}{|\Sigma|}$.
\end{proof}


As a immediate consequence of the above proposition we have:

\begin{corollary}
A two-sided closed surface with nonnegative mean curvature in the deSitter-Schwarzshild manifold $(\mathbb{R}\times\mathbb{S}^2,g_a)$ is a critical point of the Hawking mass if and only if is minimal or is a slice $\{r\}\times\mathbb{S}^2$. 
\end{corollary}

We now turn to prove Theorem \ref{theorem1}. First we prove


\begin{proposition}\label{proposition2}
Let $(\mathbb{R}\times\mathbb{S}^2,g_a)$ be the deSitter-Schwarzschild manifold with mass $m_a$ and let $\Sigma_r=\{r\}\times\mathbb{S}^2$. Then, there exists a constant $C=C(\Sigma_r)>0$ such that for all smooth normal variation $\Sigma(t)$ of $\Sigma_r$ 
\[
\left.\ddt\right|_{t=0} m_\text{\rm H}(\Sigma(t))\leqslant -C\int_{\Sigma_r}(\varphi-\overline{\varphi})^2\,\dsigma_r,
\]
where $\varphi\in C^{\infty}(\Sigma_r)$ is the function which gives the variation and $\overline{\varphi}=\frac{1}{|\Sigma_r|}\int_{\Sigma_r}\varphi\,\dsigma.$
\end{proposition}
\begin{proof}
First, we have (see Appendix):
\begin{align*}
\left.\ddt\right|_{t=0}m_\text{\rm H}(\Sigma(t))&=-\frac{|\Sigma|^{1/2}}{32\pi^{3/2}}\int_{\Sigma_r}(\Delta\varphi)^2\,\dsigma_r\\
&\hspace{-3.0cm}+\frac{1}{4\pi^{1/2} |\Sigma|^{1/2}}\int_{\Sigma_r}|\D\varphi|^2\,\dsigma-\dfrac{3\,m_a}{2|\Sigma_r|}\int_{\Sigma_r}|\D\varphi|^2\,\dsigma_r\\
&+\dfrac{3\,m_a}{4|\Sigma_r|}H^2\int_{\Sigma_r}(\varphi-\overline{\varphi})^2\,\dsigma_r,
\end{align*}
where $H$ is the mean curvature of $\Sigma_r$ and the gradients and Laplacians are computed on $\Sigma_r$.

Next, by the B\"ochner-Weitzenb\"ock identity applied on $\Sigma_r$
\begin{align*}
\dfrac{1}{2}\Delta|\D\varphi|^2&=|\He\varphi|^2+\langle\D\Delta\varphi,\D\varphi\rangle+\Ric(\D\varphi,\D\varphi)\\
&\geqslant\dfrac{1}{2}(\Delta\varphi)^2+\langle\D\Delta\varphi,\D\varphi\rangle+K_{\Sigma_r}|\D\varphi|^2\\
&=\dfrac{1}{2}(\Delta\varphi)^2+\langle\D\Delta\varphi,\D\varphi\rangle+\dfrac{4\pi}{|\Sigma_r|}|\D\varphi|^2,
\end{align*}
which once we integrate over $\Sigma_r$ we have 
$$-\frac{1}{2}\int_{\Sigma_r}(\Delta\varphi)^2\,\dsigma_r\leqslant-\frac{4\pi}{|\Sigma_r|}\int_{\Sigma_r}|\D\varphi|^2\,\dsigma_r.$$
This in turn imply
\[
\left.\ddt\right|_{t=0}m_\text{\rm H}(\Sigma(t))\leqslant-\dfrac{3}{2}\dfrac{m_a}{|\Sigma_r|}\int_{\Sigma_r}|\D\varphi|^2\,\dsigma_r+\dfrac{3}{4}\dfrac{m_a}{|\Sigma_r|}H^2\int_{\Sigma_r}(\varphi-\overline{\varphi})^2\,\dsigma_r.
\]
Moreover, since $g|{\Sigma_r}=u(r)^2g_{\mathbb{S}^2}$, we have by the Poincar\'e inequality
\begin{align*}
\int_{\Sigma_r}|\D\varphi|^2\,\dsigma_r&\geqslant\dfrac{2}{u(r)^2}\int_{\Sigma_r}(\varphi-\overline{\varphi})^2\,\dsigma_r\\
&=\dfrac{8\pi}{|\Sigma_r|}\int_{\Sigma_r}(\varphi-\overline{\varphi})^2\,\dsigma_r,
\end{align*}
and therefore we have
\[
\left.\ddt\right|_{t=0} m_\text{\rm H}(\Sigma(t))\leqslant-12\pi\dfrac{m_a}{|\Sigma_r|^2}\int_{\Sigma_r}(\varphi-\overline{\varphi})^2\,\dsigma_r+\dfrac{3}{4}\dfrac{m_a}{|\Sigma_r|}H^2\int_{\Sigma_r}(\varphi-\overline{\varphi})^2\,\dsigma_r.
\]

Finally, we note that since $H^2=4\frac{u^\prime(r)^2}{u(r)^2}$ and $u^\prime(r)^2<1$ we have that $H^2=\dfrac{16\pi}{|\Sigma_r|}-C$, where $C=C(\Sigma_r)>0$ is a positive constant, and thus
\[
\left.\ddt\right|_{t=0}m_\text{\rm H}(\Sigma(t))\leqslant-C\int_{\Sigma_r}(\varphi-\overline{\varphi})^2\,\dsigma_r.
\]
\end{proof}


\begin{proof}[Proof of Theorem \ref{theorem1}]
To prove Theorem \ref{theorem1} we will use an argument adap\-ted from \cite{BrM} and \cite{DPM}. Suppose ${\Sigma}$ is a graph over a slice $\Sigma_r$ given by a function $\varphi\in C^2(\Sigma_r)$. Assume the average $\bar{\varphi}$ of $\varphi$ is zero and let $\mathcal{L}$ be the operator given by the second variation of the Hawking mass:
\begin{align*}
\langle\mathcal{L}\varphi,\varphi\rangle=&-\dfrac{|\Sigma_r|^{1/2}}{32\pi^{3/2}}\int_{\Sigma_r}(\Delta\varphi)^2\,\dsigma_r+\dfrac{1}{4\pi^{1/2}|\Sigma_r|^{1/2}}\int_{\Sigma_r}|\D\varphi|^2\,\dsigma_r\\
&-\dfrac{3}{2}\dfrac{m_a}{|\Sigma_r|}\int_{\Sigma_r}|\D\varphi|^2\,\dsigma_r+\dfrac{3}{4}\frac{m_a}{|\Sigma_r|} H^2 \int_{\Sigma_r}\varphi^2\,\dsigma_r,
\end{align*}
By the computation in Proposition \ref{secondvariation}, we have
\begin{equation}\label{massdiff}
m_\textrm{H}({\Sigma})-m_\textrm{H}(\Sigma_r)=  \frac{1}{2}\langle \mathcal{L\varphi,\varphi}\rangle + O(||\varphi||_{C^2}||\varphi||^2_{W^{2,2}}),
\end{equation} 
where the constant in the Big-$O$ notation is uniform in $\varphi$, i.e., depends only on the slice $\Sigma$, and $W^{k,p}$ is usual notation for the Sobolev spaces.

We next claim that there must exist a constant $C>0$ such that for any function $h$ of zero average:
\begin{equation}\label{ineqstab}
\left|\langle \mathcal{L}h,h\rangle\right| \geq C ||h ||^2_{W^{2,2}}.
\end{equation}
We prove the above by contradiction: assuming the contrary, there will exist a sequence of functions $h_n$ such that 
$$ ||h_n ||^2_{W^{2,2}}=1,\qquad  \left|\langle \mathcal{L}h_n,h_n\rangle\right| <\frac{1}{n}. $$
By the Rellich-Kondrachov theorem, up to subsequence, $h_n$ must converge in $W^{1,2}$ to a limit $h$ with zero average. We would like to conclude that $ \left|\langle \mathcal{L}h,h\rangle\right| =0$, but for that we would need $h\in W^{2,2}$. So we argue as follows. First, by Proposition \ref{secondvariation}, we note that $ \left|\langle \mathcal{L}\cdot,\cdot \rangle\right| $ controls the $L^2$-norm, and since  $\left|\langle \mathcal{L}h_n,h_n\rangle\right|\rightarrow 0$, we have that $h_n$ converges to zero in $L^2$, and therefore $h$ must equal to zero. Finally, by the definition of $\mathcal{L}$, we observe that because there exists positive constants $C_1,C_2$ independent of $n$ such that:
$$\left|\langle \mathcal{L}h_n,h_n\rangle\right|\geq C_1 ||\Delta h_n||^2_{L^2} - C_2 ||h_n||_{W^{1,2}},$$
so $\Delta h_n$ must converge to zero in $L^2$, and therefore by elliptic regularity $||h_n||_{W^{2,2}}\rightarrow 0$, which is a contradiction since $||h_n ||_{W^{2,2}}=1$, and the claim follows.

Hence, combining \eqref{massdiff} and  \eqref{ineqstab}, we have for functions $\varphi$ of zero average and sufficiently small $C^2$-norm that
$$m_\textrm{H}({\Sigma})-m_\textrm{H}(\Sigma_r) \geq \frac{C}{4} ||\varphi||^2_{W^{2,2}},$$
and, by changing the argument {\it mutatis mutandis}, we have more ge\-nerally that for any $\varphi$ such that $\varphi-\bar{\varphi}$ has sufficiently small $C^2$-norm:
$$m_\textrm{H}({\Sigma})-m_\textrm{H}(\Sigma_r) \geq \frac{C}{4} ||\varphi-\bar{\varphi}||^2_{W^{2,2}}$$
and this concludes our proof.
\end{proof}


\section{Stability and second variation of the Hawking mass}

Given a surface $\Sigma$ in a three-manifold $(M,g)$, the Jacobi operator of $\Sigma$, denoted by $L_\Sigma$, or just by $L$ if there is no ambiguity, is defined to be 
\[
L=\Delta_\Sigma+\Ric(\nu,\nu)+|A|^2,
\]
where $\nu$ and $A$ denote the unit normal vector field along $\Sigma$ and the second fundamental form of $\Sigma$, respectively. We denote by $\lambda_1(L)$ the first eigenvalue of $L$. Our convention for the eigenvalue problem is the following:

\begin{center}
$\lambda\in\mathbb{R}$ is an eigenvalue of $L$ $\Leftrightarrow$ $\exists$ $\varphi\in C^{\infty}(\Sigma)$ such that  $L\varphi+\lambda\varphi=0.$
\end{center}

We start by proving a sharp upper bound involving $\lambda_1(L)$ for the area of a stable minimal two-sphere $\Sigma$ on a three-manifold $(M,g)$ with $R\geqslant2$. In case $\lambda_1(L)=0$, it is precisely the area bound that appear in \cite{BBN}. In case $\lambda_1(\Sigma)>0$, the area bound below is achieved on stricly minimal slices in the deSitter-Schwarzschild manifold $(\mathbb{R}\times\mathbb{S}^2,g_a)$.


\begin{proposition}\label{upperbound}
Let $(M,g)$ be a Riemannian three-manifold with scalar curvature $R\geqslant 2$. If $\Sigma\subset M$ is a stable minimal two-sphere, then 
\begin{equation}\label{eq:area}
|\Sigma| \leqslant \frac{4\pi}{\lambda_1(L)+1}.
\end{equation}
\end{proposition}

\begin{proof}
By the stability inequality we have that 
\[
\lambda_1(L)\int_\Sigma\varphi^2\,\dsigma+\int_\Sigma(\Ric(\nu,\nu)+|A|^2)\varphi^2\,\dsigma\leqslant\int_\Sigma|\D_\Sigma\varphi|^2\,\dsigma
\]
for all $\varphi\in C^{\infty}(\Sigma)$, where $\dsigma$ denotes the area element of $\Sigma$, and  $\lambda_1(L)\geqslant 0$.
Choosing $\varphi=1$, we get
\begin{eqnarray}\label{ineq1}
\lambda_1(L)|\Sigma|+\int_\Sigma(\Ric(\nu,\nu)+|A|^2)\,\dsigma\leqslant 0,\label{Des 1}
\end{eqnarray}
where $|\Sigma|$ is the area of $\Sigma$. The Gauss equation implies
\begin{eqnarray}
\Ric(\nu,\nu)=\dfrac{R}{2}-K_\Sigma -\dfrac{|A|^2}{2},\label{Eq 1}
\end{eqnarray}
where $K_\Sigma$ is the Gauss curvature of $\Sigma$. Substituting \eqref{Eq 1} in \eqref{Des 1}:\begin{eqnarray}\label{ineq2}
\lambda_1(L)|\Sigma|+\dfrac{1}{2}\int_\Sigma(R+|A|^2)\,\dsigma\leqslant\int_\Sigma K_\Sigma\,\dsigma=4\pi,
\end{eqnarray}
and using in \eqref{ineq2} that $R\geqslant 2$, we finally obtain
\begin{equation*}\label{ineqarea1}
|\Sigma|\leqslant \dfrac{4\pi}{\lambda_1(L)+1}.
\end{equation*}
\end{proof}

As a corollary of the proof above, we have that if the upper area bound is achieved then we get an infinitesimal rigidity over $\Sigma$.

\begin{corollary}\label{cor}
If we have equality in the above proposition, then on $\Sigma$ we must have $A=0$, $R=2$, $\Ric(\nu,\nu)=-\lambda_1(L)$, $K_\Sigma={4\pi}/{|\Sigma|}$ and $\text{\rm Ker}(L+\lambda_1(L))$ are the constant functions. 
\end{corollary}

Our next proposition gives a relation between strict stability and the Haw\-king mass. More precisely, it tells us that if the second variation of the Hawking mass of a stricly stable minimal two-sphere $\Sigma$ is non-positive for all smooth normal variations $\Sigma(t)$ of $\Sigma$, then we get the reverse inequality in \eqref{upperbound}. We therefore get equality in \eqref{upperbound} and the conclusions of Corollary \ref{cor} follows in this case.

 Recall that, by definition, $\Sigma$ is stricly stable when  $\lambda_1(L)>0.$


\begin{proposition}\label{lowerbound}
Let $(M,g)$ be a Riemannian three-manifold with scalar curvature $R\geqslant 2$ and let $\Sigma\subset M$ be a minimal two-sphere. If $\Sigma$ is strictly stable and the second variation of the Hawking mass of $\Sigma$ is non-positive, then
\begin{equation}\label{eq:mass}
|\Sigma| \geqslant \frac{4\pi}{\lambda_1(L)+1}.
\end{equation}
\end{proposition}
\begin{proof}
Let $\Sigma(t)\subset M$ be a smooth normal variation of $\Sigma$ given by a vector field $X=\varphi\nu$, where $\varphi\in C^{\infty}(\Sigma)$. Since $\Sigma$ is a minimal surface, a direct computation gives
\begin{eqnarray*}
\left.\ddt \right|_{t=0}m_\text{\rm H}(\Sigma(t))&=&-\dfrac{1}{128\,\pi^{3/2}|\Sigma|^{1/2}}\int_\Sigma\varphi L\varphi\,\dsigma\pl16\pi-\dfrac{4}{3}|\Sigma|\pr\\
&+&\dfrac{|\Sigma|^{1/2}}{64\,\pi^{3/2}}\pl-2\int_\Sigma(L\varphi)^2\,\dsigma+\dfrac{4}{3}\int_{\Sigma}\varphi L\varphi\,\dsigma\pr,
\end{eqnarray*}
and, because $\left.\ddt\right|_{t=0}m_\text{\rm H}(\Sigma(t))\leqslant 0$, we get that
\begin{equation}\label{desigualdade3}
(8\pi-2|\Sigma|)\pl-\int_\Sigma\varphi L\varphi\,\dsigma\pr\leqslant2|\Sigma|\int_\Sigma(L\varphi)^2\,\dsigma.
\end{equation}

Furthermore, if we apply in \eqref{desigualdade3} an eigenfunction of $\lambda_1(L)$ satisfying $\int_\Sigma\varphi^2\,\dsigma=1$, we obtain
\[
(8\pi-2|\Sigma|)\lambda_1(L)\leqslant2|\Sigma|\lambda_1(L)^2,
\]
and, since $\lambda_1(L)>0$, this in turn imply
\[
(8\pi-2|\Sigma|)\leqslant2|\Sigma|\lambda_1(L),
\]
and the result follows.
\end{proof}



\begin{remark}
When a minimal two-sphere is stable but not strictly so, i.e.,  in case $\lambda_1(L)= 0$, one cannot use the hypothesis  of Proposition \ref{lowerbound} to conclude the infinitesimal rigidity of Corollary \ref{cor}. In this case, the correct assumption to make in order to have rigidity is the one made in \cite{BBN}, that is, to bypass Proposition \ref{lowerbound} and assume directly that $\Sigma$ is an area-minimizing two-sphere satisfying $|\Sigma|=4\pi$, and in this case $\Sigma$ is in fact a global maximum of the Hawking mass.
\end{remark}

\section{Proof of Theorem \ref{theorem2}}

Let $(M,g)$ be a Riemannian three-manifold and consider a two-sided compact surface $\Sigma\subset M$. If $\Sigma$ is a strictly stable minimal surface we can always use the implicit function theorem to find a smooth function
$w:(-\epsilon,\epsilon)\times\Sigma\longrightarrow\mathbb{R}$ with $w(0,x)=0$, $\forall\, x\in\Sigma$, such that the surfaces
\[\Sigma(t)=\{\exp_x(w(t,x)\nu(x)):x\in\Sigma\},\,t\in(-\epsilon,\epsilon),
\] have constant mean curvature, where $\nu$ is the unit normal vector field along $\Sigma$ and $\exp$ is the exponential map of $M$. But if we do not have any other information on $\Sigma$, we cannot conclude that the one-parameter family $\Sigma(t)$ of surfaces defined above gives a foliation of a neighborhood of $\Sigma$ in $M$ because $\frac{\partial w }{\partial t}(0,\cdot)$ may change sign.

Now suppose that $(M,g)$ has scalar curvature  $R\geqslant 2$ and that $\Sigma\subset M$ is an embedded strictly stable minimal two-sphere. In adittion, suppose that the second variation at $t=0$ of the Hawking mass of all smooth normal variations  $\Sigma(t)$ of $\Sigma$ is non-positive. Then, in this case, from propositions \ref{upperbound} and \ref{lowerbound}, we have the infinitesimal rigidity, i.e., 
\[
|\Sigma|=\dfrac{4\pi}{\lambda_1(L)+1},
\]
and the conclusions of Corollary \ref{cor} holds.  It will follow from this that we can construct a one-parameter family $\Sigma(t)$ as described above, with the function $w$ satisfying $\frac{\partial w}{\partial t}(0,\cdot)=1$, and the family $\Sigma(t)$ defined using this function $w$ giving a foliation of a neighborhood of $\Sigma$ by CMC embedded two-spheres. This is proved in the next proposition.

\begin{proposition}\label{cmcfoliation}
Let $(M,g)$ be a Riemannian three-manifold with scalar curvature  $R\geqslant 2$. If $\Sigma\subset M$ is an embedded stable minimal two-sphere such that
\[
|\Sigma|=\dfrac{4\pi}{\lambda_1(L)+1},
\]
then there exist $\epsilon>0$ and a smooth function $w:(-\epsilon,\epsilon)\times\Sigma\longrightarrow\mathbb{R}$ satisfying the following conditions:
\begin{itemize}
\item For each $t\in(-\epsilon,\epsilon)$, $\Sigma(t)=\{\exp_x(w(t,x)\nu(x)):x\in\Sigma\}$ is an embedded two-sphere with constant mean curvature.
\item $w(0,x)=0$, $\dfrac{\partial w}{\partial t}(0,x)=1$ and $\int_\Sigma(w(t,\cdot)-t)\,\dsigma=0.$
\end{itemize}
\end{proposition}
\begin{proof}
The proof follows along the same lines as the proof of Proposition 2 in \cite{N}. We use the same notations used there. 

We consider the map $\Psi:(-\epsilon,\epsilon)\times B(0,\delta)\longrightarrow Y$ defined by
\[
\Psi(t,u)=H_{\Sigma_{u+t}}-\dfrac{1}{|\Sigma|}\int_\Sigma H_{\Sigma_{u+t}}\,\dsigma,
\]
and we notice that $\Psi(0,0)$, because $\Sigma_0=\Sigma$. By Corollary \ref{cor}, we have that the Jacobi operator of $\Sigma$ is given by
\[
L=\Delta_\Sigma-\lambda_1(L).
\]
Thus, obtain for $v\in X$ that
\begin{align*}
D\Psi(0,0)\cdot v&=\left.\dfrac{\textrm{d}}{\textrm{d}s}\right|_{s=0}\Psi(0,s)\\
&=\left.\dfrac{\textrm{d}}{\textrm{d}s}\right|_{s=0}\pl H_{\Sigma_{sv}}-\dfrac{1}{|\Sigma|}\int_\Sigma H_{sv}\,\dsigma\pr\\
&=Lv+\dfrac{\lambda_1(L)}{|\Sigma|}\int_\Sigma v\,\dsigma\\
&=Lv,
\end{align*}
and since $L:X\longrightarrow Y$ is a linear isomorphism, we can use the implicit function theorem to find the function $w:(-\epsilon,\epsilon)\times\Sigma\longrightarrow \mathbb{R}$ as in \cite{N}.

Moreover, it is easy to see that $w$ satisfies $w(0,\cdot)=0$ and $\int_\Sigma(w(t,\cdot)-t)\,\dsigma=0$, and that the latter implies
\[
\int_\Sigma\dfrac{\partial w}{\partial t}(0,\cdot)\,\dsigma=|\Sigma|.
\] 
Furthermore, since $H_{\Sigma_{w(t\cdot)}}=\frac{1}{|\Sigma|}\int_\Sigma H_{\Sigma_{w(t,\cdot)}}\,\dsigma$, $\forall t\in(-\epsilon,\epsilon)$, we have after differentiating at $t=0$ that
\begin{align*}
L\pl\dfrac{\partial w}{\partial t}(0,\cdot)\pr&=\dfrac{1}{|\Sigma|}\int_\Sigma L\pl\dfrac{\partial w}{\partial t}(0,\cdot)\pr\,\dsigma\\
&=-\dfrac{\lambda_1(L)}{|\Sigma|}\int_\Sigma\dfrac{\partial w}{\partial t}(0,\cdot)\,\dsigma\\
&=-\lambda_1(L)\\
&=L(1),
\end{align*}
and we thus conclude that $\frac{\partial w}{\partial t}(0,\cdot)=1$, for the strict stability of $\Sigma$ implies that $L$ is injective.
\end{proof}

We are now interested in properties of the CMC foliation constructed above. We will say that a CMC surface $\Sigma$ in a three-manifold $(M,g)$ is \textit{weakly stable} if
\[
\int_\Sigma|\D_{\Sigma}\varphi|^2-(\Ric(\nu,\nu)+|A_\Sigma|^2)\varphi^2\,\dsigma\geqslant0,
\]
for all $\varphi\in C^{\infty}(\Sigma)$ such that $\int_\Sigma\varphi\,\dsigma=0$. Inspired by Lemma 3.3 of \cite{BBN}, we next prove that, decreasing $\epsilon$
if necessary, all surfaces $\Sigma(t)$ in the foliation of Proposition \ref{cmcfoliation} are weakly stable.

\begin{lemma}\label{weakstability}
Consider $(M,g)$, $\Sigma$ and $\Sigma(t)$ as in Proposition \ref{cmcfoliation}. Then, there exists  $0<\delta<\epsilon$ such that:  if  $t\in(-\delta,\delta)$ and  $u$ is a function on the two-sphere with $\int_{\Sigma(t)}u\,\dsigma_t=0$, then
$$
\int_{\Sigma(t)}|\nabla_{\Sigma(t)}u|^2\,\dsigma_t-\int_{\Sigma(t)}(\Ric(\nu_t,\nu_t)+|A_{\Sigma(t)}|^2)u^2\,\dsigma_t\geqslant \lambda_1(L_\Sigma)\int_{\Sigma(t)}u^2\,\dsigma_t,
$$
where $\nu_t$ is the unit normal vector field along $\Sigma(t)$ with $\nu_0=\nu$.
\end{lemma}
\begin{proof}
We start by noting that a uniform constant $C>0$ can be chosen such that the Poincar\'{e} inequality 
\[
\int_{\Sigma(t)}|\nabla_{\Sigma(t)}u|^2\,\dsigma_t \geqslant C  \int_{\Sigma(t)}u^2\,\dsigma_t
\]
holds for each $t\in(-\epsilon,\epsilon)$  and any smooth function function $u:\Sph^2\longrightarrow \R$ such that $\int_{\Sigma(t)}u\,\dsigma_t=0$. In addition, when $t=0$ we know by assumption that $\Sigma(t)$ satisfies the hypothesis of Corollary \ref{cor} and thus
$$\displaystyle \sup_{\Sigma(t)} (\Ric(\nu_t,\nu_t)+|A_{\Sigma(t)}|^2+\lambda_1(L_\Sigma)) \rightarrow 0$$
as $t\rightarrow 0$. These two facts together produce the desired $\delta$.
\end{proof}

Again, let $(M,g)$, $\Sigma$ and $\Sigma(t)$ as in Proposition \ref{cmcfoliation}. We introduce  some notation. Let $f(t,x)=\exp_x(w(t,x)\nu(x))$, $(t,x)\in(-\delta,\delta)\times\Sigma$, where $\delta>0$ is given by Lemma \ref{weakstability}. Consider the \textit{lapse function}
\[
\rho_t(x)=\left\la\dfrac{\partial f}{\partial t}(t,x),\nu_t(x)\right\ra,\,(t,x)\in(-\delta,\delta)\times\Sigma.
\]
Since $\rho_0=1$, we can assume, decreasing $\delta>0$ if necessary, that $\rho_t>0$. Finally, denote by $H_t$ the mean curvature of $\Sigma(t)$ with respect to $\nu_t$ and let $\overline{\rho_t}=\frac{1}{|\Sigma(t)|}\int_{\Sigma(t)}\rho_t\,\dsigma_t.
$

Now, we are in a position to state and prove our next lemma.
\begin{lemma}\label{monotonicity}
For each $t\in(\delta,\delta)$, we have 
\begin{eqnarray*}
\int_{\Sigma(t)}(\Ric(\nu(t),\nu(t))+|A_{\Sigma(t)}|^2)\,\rho_t\,\dsigma_t\geqslant \dfrac{\lambda_1(L_\Sigma)}{\overline{\rho_t}}\int_{\Sigma(t)}(\rho_t-\rhobt)^2\,\dsigma_t \\
+\rhobt\int_{\Sigma(t)}(\Ric(\nu(t),\nu(t))+|A_{\Sigma(t)}|^2)\,\dsigma_t.
\end{eqnarray*}
\end{lemma}
\begin{proof}
The result follows from the fact that $\dt H_t =L_{\Sigma(t)} \rho_t$ together with the weak stability inequality of lemma \ref{weakstability}.  In fact,  since $\rho_t-\rhobt$ has zero average on $\Sigma(t)$, we have for each $t\in(-\delta,\delta)$ that 
\begin{align*}
\lambda_1(L_\Sigma)\int_{\Sigma(t)}(\rho_t-\overline{\rho_t})^2\,\dsigma_t&\leqslant-\int_{\Sigma(t)}(\rho_t-\overline{\rho_t})L_{\Sigma(t)}(\rho_t-\overline{\rho_t})\,\dsigma_t\\
&\hspace{-1.0cm}=-\dt H_t\int_{\Sigma(t)}(\rho_t-\overline{\rho_t})\,\dsigma_t+\int_{\Sigma}(\rho_t-\overline{\rho_t})L_{\Sigma(t)}\overline{\rho_t}\,\dsigma_t\\
&\hspace{-0.8cm}=\int_{\Sigma(t)}(\rho_t-\overline{\rho_t})\,(\Ric(\nu_t,\nu_t)+|A_{\Sigma(t)}|^2)\,\overline{\rho_t}\,\dsigma
\end{align*}
and this proves the lemma.
\end{proof}

The proof of theorem \ref{theorem2} is now mostly a matter of putting these facts together.
\begin{proof}[Proof of Theorem \ref{theorem2}]
Let $(M,g)$ and $\Sigma=\mathbb{S}^2\subset M$ satisfying our assumptions. Since $\Sigma$ is a local maximum for the Hawking mass, we have: 
$$
\left.\ddt\right|_{t=0}m_\text{\rm H}(\Sigma(t))\leqslant 0,
$$
for all smooth normal variations $\Sigma(t)$ of $\Sigma$, and by Corollary \ref{cor}:
$$
|\Sigma|=\frac{4\pi}{\lambda_1(L)+1}.
$$

By Proposition \ref{cmcfoliation}, we can construct a CMC foliation of a neighborhood of $\Sigma$ in $M$  by embedded two-spheres $\Sigma(t)\subset M$, with $t\in(-\epsilon,\epsilon)$. 

Noting that 
$$
\left.\dt\right|_{t=0}H_t=L(1)=-\lambda_1(L_\Sigma)<0,
$$
and decreasing $\epsilon>0$ if necessary, we can assume that $H_t<0$ for $t\in(0,\epsilon)$ and that $H_t>0$ for $t\in(-\epsilon,0)$. We can also assume that $m_\text{\rm H}(\Sigma)\geqslant m_\text{\rm H}(\Sigma(t))$ for $t\in(-\epsilon,\epsilon)$ because $\Sigma$ is a local maximum for the Hawking mass.

Now, let $\delta>0$ be given by Lemma \ref{weakstability} so that for each $t\in(-\delta,\delta)$, $\Sigma(t)\subset M$ is a weakly stable CMC two-sphere. In what follows, we will see that this implies, using Lemma \ref{monotonicity}, monotonicity of the Hawking mass along the foliation $\Sigma(t)$. 

In fact, we have
\begin{align*}
\dt m_\text{\rm H}(\Sigma(t))&=-\dfrac{|\Sigma(t)|^{1/2}}{32\pi^{3/2}}H_t\left[\int_{\Sigma(t)}\pl\Ric(\nu_t,\nu_t)+|A_{\Sigma(t)}|^2\pr\rho_t\,\dsigma_t\right.\\
&\left.+4\pi\overline{\rho_t}-\dfrac{3}{4}H_t^2\int_{\Sigma(t)}\rho_t\,\dsigma_t-\int_{\Sigma(t)}\rho_t\,\dsigma_t\right]\\
\phantom{\dt m_\text{\rm H}(\Sigma(t))}&\geqslant-\dfrac{|\Sigma(t)|^{1/2}}{32\pi^{3/2}}H_t\left[\dfrac{\lambda_1(L_\Sigma)}{\overline{\rho_t}}\int_{\Sigma(t)}\pl\rho_t-\overline{\rho_t}\pr^2\,\dsigma_t\right.+\\
&\hspace{-2.3cm}\left.\overline{\rho_t}\int_{\Sigma(t)}\pl\Ric(\nu_t,\nu_t)+|A_{\Sigma(t)}|^2\pr\,\dsigma_t+4\pi\overline{\rho_t}-\dfrac{3}{4}H_t^2\int_{\Sigma(t)}\rho_t\,\dsigma_t-\int_{\Sigma(t)}\rho_t\,\dsigma_t\right],
\end{align*}
where the inequality follows by Lemma \ref{monotonicity}, and moreover, using the Gauss equation: 
\begin{align*}
\dt m_\text{\rm H}(\Sigma(t))&\geqslant-\dfrac{|\Sigma(t)|^{1/2}}{32\pi^{3/2}}H_t\left[\frac{\overline{\rho_t}}{2}\int_{\Sigma(t)}\pl|A_{\Sigma(t)}|^2-\dfrac{H_t^2}{2}\pr+(R-2)\,\dsigma_t\right.\\
&\left.+\dfrac{\lambda_1(L_\Sigma)}{\overline{\rho_t}}\int_{\Sigma(t)}\pl\rho_t-\overline{\rho_t}\pr^2\,\dsigma_t\right].
\end{align*} 

Thus, by the formula above, we obtain that $\dt m_\text{\rm H}(\Sigma(t))\geqslant 0$ for $t\in[0,\delta)$ and $\dt m_{H}(\Sigma(t))\leqslant 0$ for $t\in(-\delta,0].$ This implies that
\[
m_\text{\rm H}(\Sigma)\leqslant m_\text{\rm H}(\Sigma(t)),
\]
for all $t\in(-\delta,\delta)$. Since $m_\text{\rm H}(\Sigma)\geqslant m_\text{\rm H}(\Sigma(t))$, we conclude that $m_\text{\rm H}(\Sigma(t))\equiv m_\text{\rm H}(\Sigma)$ and so $\dt m_\text{\rm H}(\Sigma(t))\equiv 0$, and from this, using the formulae above, we have for all $t\in(-\delta,\delta)$ that
\begin{itemize}
\item $\Sigma(t)$ is umbilic;
\item $R=2$ on $\Sigma(t)$;
\item $\rho_t\equiv\overline{\rho_t}$.
\end{itemize}
Moreover, using that $\rho_t\equiv\overline{\rho_t}$, it is not difficult to show that 
$$
w(t,x)=t,\, \forall (t,x)\in(-\delta,\delta)\times\Sigma.
$$ 
Finally, denote by $g_{\Sigma(t)}$ the induced metric on $\Sigma(t)$. Since $\Sigma(t)$ is umbilic and $H_t$ is constant, we have  
$$
\dfrac{\partial}{\partial t}g_{\Sigma(t)}=v(t)g_{\Sigma(t)},\,\forall t\in(-\delta,\delta),
$$ 
where $v$ is a real function. Thus, we get for all $t\in(-\delta,\delta)$ that
\begin{align*}
g_{\Sigma(t)}&=e^{\int_{0}^{t}v(s)\,\textrm{d}s}\,g_\Sigma\\
&=u(t)^2\,g_{\mathbb{S}^2},
\end{align*}
where $u(t)=a\,e^{\int_{0}^{t}v(s)\,\textrm{d}s}$ with $a^2=|\Sigma|/4\pi\in(0,1)$.

Therefore, we conclude that the metric $\overline{g}$ on $(-\delta,\delta)\times\Sigma$ induced by $f(t,x)=\exp_x(t\nu(x))$, $(t,x)\in(-\delta,\delta)\times\Sigma,$ is equal to $dt^2+u(t)^2g_{\mathbb{S}^2}$. Since this metric has scalar curvature equal to 2, we have, by unicity of solutions to \eqref{EDO1}, that $\overline{g}$ is precisely the deSitter-Schwarzschild metric with mass $m_a$ on $(-\delta,\delta)\times\Sigma.$ This finishes the proof.
\end{proof} 

\section{appendix}



Let $(M,g)$ be a three-manifold and consider a two-sided compact surface $\Sigma\subset M$. Our goal in this section is to provide the first and second variation formulae of the Hawking mass at $\Sigma$. Recall that the Hawking mass is defined by
$$
m_\text{\rm H}(\Sigma)=\pl\frac{|\Sigma|}{16\pi}\pr^{1/2}\pl1-\frac{1}{16\pi}\int_\Sigma H^2\,\dsigma-\frac{\Lambda}{24\pi}|\Sigma|\pr,
$$
where $\Lambda=\inf R$.

Choose a unit normal vector field $\nu$ along $\Sigma$ and let $\Sigma(t)\subset M$  be a smooth normal variation of $\Sigma$, that is, $\Sigma(t)=\{f(t,x):x\in\Sigma\}$ where $f:(-\epsilon,\epsilon)\times\Sigma\longrightarrow M$ is a smooth function satisfying:
\begin{itemize}
\item $f_t=f(t,\cdot):\Sigma\longrightarrow M$ is an immersion for each $t\in(-\epsilon,\epsilon)$;
\item $f(0,x)=x$ for each $x\in\Sigma$;
\item $\frac{\partial f}{\partial t}(0,x)=\varphi(x)\nu(x)$ for each $x\in\Sigma$, where $\varphi\in C^\infty(\Sigma).$
\end{itemize}

For such a given variation, we have:

\begin{proposition}[First variation of the Hawking mass]\label{firstvariation} 
\begin{align*}
\left. \dt m_\text{\rm H}(\Sigma(t))\right|_{t=0}&=-\frac{2|\Sigma|^{1/2}}{(16\pi)^{3/2}}\int_\Sigma \varphi\Delta_\Sigma H\,\dsigma\\
&\hspace{-3.0cm}+\frac{|\Sigma|^{1/2}}{(16\pi)^{3/2}}\int_\Sigma\left[2K_\Sigma-\dfrac{8\pi}{|\Sigma|}+\pl\dfrac{1}{2|\Sigma|}\int_\Sigma H^2\,\dsigma-|A|^2\pr\right]H\varphi\,\dsigma\\
&+\frac{|\Sigma|^{1/2}}{(16\pi)^{3/2}}\int_\Sigma(\Lambda-R)H\varphi\,\dsigma.
\end{align*}
\end{proposition}
\begin{proof}
This is a direct computation using the first variation formula of area and the following identities:
\begin{itemize}
\item[(i)] $\left.\dt H_{t}\right|_{t=0}=\Delta_\Sigma\varphi+\Ric(\nu,\nu)\varphi+|A|^2\varphi$;
\item[(ii)] $\left.\dt (\dsigma_t)\right|_{t=0}=-\varphi H\dsigma.$ 
\end{itemize}
and the Gauss equation $2\Ric(\nu,\nu)=R-2K_\Sigma+H^2-|A|^2$.  For identities (i) and (ii) see \cite{HP}.

\end{proof}


Now, denote by $\nu_t$ the unit normal vector along $\Sigma(t)$ with $\nu_0=\nu$ and let $H_t$ be the mean curvature of $\Sigma(t)$ with respect to $\nu_t$. Consider the \textit{lapse function}
$$
\rho_t(x)=\left\la\frac{\partial f}{\partial t}(t,x),\nu_t(x)\right\ra.
$$
Notice that $\rho_0=\varphi.$ Also, it is a well-known fact that
\begin{equation}\label{derivadadeH}
\dt H_t=L(t)\rho_t,
\end{equation}
where 
$$
L(t)=\Delta_{\Sigma(t)}+\Ric(\nu_t,\nu_t)+|A_{\Sigma(t)}|^2
$$
is the Jacobi operator of $\Sigma(t)$. By $\eqref{derivadadeH}$, we have the second variation formula of the mean curvature  $H_t$ at $t=0$ as a consequence of the next proposition.

\begin{proposition}[First variation of the Jacobi operator]\label{firstvariationofL}
For each function $\psi\in C^\infty(\Sigma)$, we have:
\begin{align*}
L^\prime(0)\,\psi&=2\,\varphi\,\la A,\He\psi\ra+2\,\psi\,\la A,\He\varphi\ra-2\,\varphi\,\omega(\D\psi)-2\,\psi\,\omega(\D\varphi)\\
+&\varphi\,\la\D H,\D\psi\ra-H\,\la\D\varphi,\D\psi\ra+2A(\D\varphi,\D\psi)-\psi\,\di_\Si(\di_\Si\omega)\\
+&\varphi\,\psi\,R_{i\nu\nu j}\,A_{ij}+\varphi\,\psi\,H\Ric(\nu,\nu)+\varphi\,\psi H|A|^2+\varphi\,\psi\,A_{ij}\,A_{ik}\,A_{jk}\\
-&\varphi\,\psi HK_\Sigma,
\end{align*}
where $\omega$ is the 1-form on $\Sigma$ defined by $\omega(X)=\Ric(X,\nu)$.
\end{proposition}
\begin{proof}
Using the Gauss equation, we can rewrite $L(t)$ as 
$$
L(t)=\Delta_{\Sigma(t)}+\frac{R}{2}-K_{\Sigma(t)}+\frac{H_t^2}{2}-\frac{|A_{\Sigma(t)}|^2}{2}.
$$

Now, the formula follows by a straightforward computation using the following identities. The first one is
\begin{align*}
\left.\dt K_{\Sigma(t)}\right|_{t=0}&=-\la A,\He\varphi\ra+H\,\Delta\varphi+2\,\omega(\D\varphi)\\
&+\di_\Si(\di_\Si\omega)\,\varphi+H\,K_{\Sigma}\,\varphi,
\end{align*}
which can be derived directly from Lemma 3.7 of \cite{CK} with $h_{ij}=-2\varphi A_{ij}$.  The other ones are
\begin{align*}
\pl\left.\dt \Delta_{\Sigma(t)}\right|_{t=0}\pr\psi&=2\,\varphi\,\la A,\He\psi\ra+2\,A(\D\varphi,\D\psi)+\varphi\,\la\D H,\D\psi\ra\\
&-H\,\la\D\varphi,\D\psi\ra-2\,\varphi\omega(\D\psi),
\end{align*}
and
\begin{align*}
\left.\dt|A_{\Sigma(t)}|^2\right|_{t=0}=2\,\la A,\He\varphi\ra+2\,R_{i\nu\nu j}\,A_{ij}\,\varphi+2\,A_{ij}\,A_{ik}\,A_{jk}\,\varphi,
\end{align*}
whose proof can be found in detail at \cite{LMS}.

\end{proof}

Next, we have the second variation of the Hawking mass.

\begin{proposition}[Second variation of the Hawking mass]\label{secondvariation}
If $\Sigma\subset M$ is a critical point of the Hawking mass, then
\begin{align*}
\left.\ddt m_\text{\rm H}(\Sigma(t))\right|_{t=0}&=-\frac{2|\Sigma|^{1/2}}{(16\pi)^{3/2}}\int_\Sigma(L\varphi)^2\,\dsigma+\frac{4|\Sigma|^{1/2}}{(16\pi)^{3/2}}\int_\Sigma H^2\varphi L\varphi\,\dsigma\\
&\hspace{-3.5cm}+\frac{m_\text{\rm H}(\Sigma)}{2|\Sigma|}\int_\Sigma|\D\varphi|^2\,\dsigma-\frac{|\Sigma|^{1/2}}{(16\pi)^{3/2}}\int_\Sigma\pl H^2+\frac{2\Lambda}{3}\pr|\D\varphi|^2\,\dsigma\\
&\hspace{-1.5cm}-\dfrac{m_\text{\rm H}(\Sigma)}{2|\Sigma|}\int_\Sigma\pl\Ric(\nu,\nu)+|A|^2-H^2\pr\varphi^2\,\dsigma\\
&\hspace{-2.0cm}+\frac{|\Sigma|^{1/2}}{(16\pi)^{3/2}}\int_\Sigma\pl H^2+\frac{2\Lambda}{3}\pr\pl \Ric(\nu,\nu)+|A|^2-H^2\pr\varphi^2\,\dsigma\\
&\hspace{-2.8cm}-\frac{3\,m_\text{\rm H}(\Sigma)}{2|\Sigma|^2}\pl\int_\Sigma H\varphi\,\dsigma\pr^2-\frac{2|\Sigma|^{1/2}}{(16\pi)^{3/2}}\int_\Sigma HL^\prime(0)\varphi\,\dsigma,
\end{align*}
where $L=L(0)$ and $L^{\prime}(0)$ is given in the proposition above.
\end{proposition}

\begin{proof}

Once establishing \eqref{derivadadeH}, the above follows after a direct computation using the second variation formula of the area element:
$$\left.\ddt (\dsigma_t )\right|_{t=0} = \left[ |\nabla \varphi|^2 - (\Ric (\nu,\nu) + |A|^2)\varphi^2 + H^2\varphi^2 +  \di_\Sigma (\nabla_X X)\right] \dsigma,$$
where $X(x)=\frac{\partial f}{\partial t} (0,x)$.
\end{proof}

To finish this section, we will consider the particular case where $(M,g)$ is the deSitter-Schwarzschild manifold $(\mathbb{R}\times\mathbb{S}^2,g_a)$ and $\Sigma\subset M$ is some slice $\{r\}\times\mathbb{S}^2$. In this case, we have:
\begin{itemize}
\item $R$ is constant equal to $2$;
\item $\Sigma$ is totally umbilic and has constant Gauss curvature;
\item $\omega=0$ (recall definition in Proposition \ref{firstvariationofL}).
\end{itemize}

Therefore, we have by \eqref{firstvariationofL} that
$$
L^\prime(0)\varphi=2H\varphi\Delta\varphi+\dfrac{3}{2}H\pl\Ric(\nu,\nu)+\dfrac{H^2}{2}\pr\varphi^2-\dfrac{4\pi}{|\Sigma|}H\varphi^2.
$$

Thus, since
$$
\Ric(\nu,\nu)+\frac{H^2}{2}=\frac{8\pi}{|\Sigma|}-(16\pi)^{3/2}\frac{3}{4}\frac{m_\text{\rm H}(\Sigma)}{|\Sigma|^{3/2}},
$$
we have after a direct but long computation using \eqref{secondvariation}
\begin{align*}
\left.\ddt\right|_{t=0}m_\text{\rm H}(\Sigma(t))&=-\frac{|\Sigma|^{1/2}}{32\pi^{3/2}}\int_{\Sigma}(\Delta\varphi)^2\,\dsigma+\frac{1}{4\pi^{1/2} |\Sigma|^{1/2}}\int_{\Sigma}|\D\varphi|^2\,\dsigma\\
&-\dfrac{3\,m_\text{\rm H}(\Sigma)}{2|\Sigma|}\int_{\Sigma}|\D\varphi|^2\,\dsigma+\dfrac{3\,m_\text{\rm H}(\Sigma)}{4|\Sigma|}H^2\int_{\Sigma}(\varphi-\overline{\varphi})^2\,\dsigma,
\end{align*}
where $\overline{\varphi}=\frac{1}{|\Sigma|}\int_\Sigma\varphi\,\dsigma$ and we have used in the above that
$$
\int_\Sigma\pl\varphi-\overline{\varphi}\pr^2\,\dsigma=\int_\Sigma\varphi^2\,\dsigma-\frac{1}{|\Sigma|}\pl\int_\Sigma\varphi\,\dsigma\pr^2.
$$

\section{Acknowledgements}
We are very grateful to Fernando C. Marques for his support and mentorship. We also thank Andr\'{e} Neves for many helpful discussions, and Karen Uhlenbeck who kindly supported our work through her Sid W Richardson Regents Foundation Chair 3 in Mathematics.~ D.M. thanks his advisor Dan Knopf for his encouragement and both authors would like to thank FAPERJ, CNPq-Brazil and NSF for their financial support.

\bibliographystyle{ams}

\end{document}